\documentclass[conference]{IEEEtran}

\usepackage[utf8]{inputenc}
\usepackage[T1]{fontenc}

\usepackage{cite}
\usepackage[cmex10]{amsmath}
\usepackage{array}
\usepackage{thmbox}
\newtheorem[L,rightmargin=6pt,thickness=1.5pt]{thmchapter}{Theorem}[section]

\usepackage{bbm}
\usepackage{bm}
\usepackage{graphicx}
\usepackage{caption}
\usepackage{subcaption}
\usepackage{amssymb,ntheorem}
\def\N{\mathbb{N}} 
\def\R{\mathbb{R}} 
\def\ell{\textit{l}} 

\def\wp1{\mathrm{w.p.} 1}  

\graphicspath{{figures/}}

\newtheorem{theoreme}{Theorem}[section]
\newtheorem*{lemme}{Useful lemmas.}
\newtheorem{proposition}[theoreme]{Proposition}

\newtheorem{definition}[theoreme]{Definition\rm}
\newtheorem*{remarque}{\it Remark}

\def\cc#1{\setlength{\tabcolsep}{0pt}\begin{tabular}{c}#1\end{tabular}}
\newcommand{\figc}[2][]
   {\setlength{\tabcolsep}{0pt}\begin{tabular}{c}\includegraphics[#1]{#2}\end{tabular}}
\def\yaxis#1{\cc{\rotatebox{90}{{\small #1}}}}

\usepackage{xcolor}
\usepackage[normalem]{ulem}

\definecolor{indigo}{rgb}{.29,0.,0.51}

\begin{document}

\title{Travelling salesman-based compressive sampling.}

\author{\IEEEauthorblockN{Nicolas Chauffert, Philippe Ciuciu}
\IEEEauthorblockA{CEA, NeuroSpin center, \\
INRIA Saclay, PARIETAL Team \\
145, F-91191 Gif-sur-Yvette, France\\
Email: firstname.lastname@cea.fr}
\and
\IEEEauthorblockN{Jonas Kahn}
\IEEEauthorblockA{Laboratoire Painlev\'e, UMR 8524\\
    Universit\'e de Lille 1, CNRS\\
 Cit\'e Scientifique B\^at. M2\\
 59655 Villeneuve d'Ascq Cedex, France\\
Email: jonas.kahn@math.univ-lille1.fr}
\and
\IEEEauthorblockN{Pierre Weiss}
\IEEEauthorblockA{ITAV, USR 3505\\
PRIMO Team, \\
Universit\'e de Toulouse, CNRS\\
Toulouse, France\\
Email: pierre.weiss@itav-recherche.fr}}
\maketitle

\begin{abstract}
Compressed sensing theory indicates that selecting a few measurements independently at random is a near optimal
strategy to sense sparse or compressible signals. 
This is infeasible in practice for many acquisition devices that acquire samples along \textit{continuous} trajectories. 
Examples include magnetic resonance imaging (MRI), radio-interferometry, mobile-robot sampling, ...
In this paper, we propose to generate continuous sampling trajectories by drawing a small set of measurements
independently and joining them using a travelling salesman problem solver.
Our contribution lies in the theoretical derivation of the appropriate probability density of the initial drawings.
Preliminary simulation results show that this strategy is as efficient as independent drawings while being
implementable on real acquisition systems.
\end{abstract}

\IEEEpeerreviewmaketitle


\section{Introduction}

Compressed sensing theory provides guarantees on the reconstruction quality of sparse and compressible signals $x\in \R^n$ from a limited number of linear measurements $(\langle a_k, x\rangle)_{k\in K}$.
In most applications, the measurement or acquisition basis $A=(a_k)_{k\in \{1,\cdots,n\}}$ is fixed (e.g. Fourier or Wavelet basis). 
In order to reduce the acquisition time, one then needs to find a set $K$ of minimal cardinality 
that provides satisfactory reconstuction results.
It is proved in~\cite{candes2006near,rauhut2010compressive} that a good way to proceed consists of drawing the indices of $K$ independently at random according to a distribution $\tilde \pi$ that depends on the sensing basis $A$.
This result motivated a lot of authors to propose variable density random sampling strategies (see e.g.~\cite{lustig2007sparse,knoll2011adapted,chauffert2013}).
Fig.~\ref{fig:Distributions}(a) illustrates a typical sampling pattern used in the MRI context. 
Simulations confirm that such schemes are efficient in practice. 
Unfortunately, they can hardly be implemented on real hardware where the physics of the acquisition processes imposes \emph{at least} continuity of the sampling trajectory and sometimes a higher level of smoothness. Hence, actual CS-MRI solutions relie on adhoc
solutions such as random radial or randomly perturbed spiral trajectories to impose gradient continuity. 
Nevertheless these strategies strongly deviate from the theoretical setting and experiments confirm their practical suboptimality.

In this work, we propose an alternative to the independent sampling scheme. 
It consists of picking a few samples independently at random according to a distribution $\pi$ and joining them using a travelling salesman problem~(TSP) solver in order to design continuous trajectories. The main theoretical result of this paper states that $\pi$ should be proportional to $\tilde \pi^{d/(d-1)}$ where $d$ denotes the space dimension~(e.g. $d=2$ for 2D images, $d=3$ for 3D images) in order to emulate an independent drawing from distribution $\tilde \pi$. Similar ideas were previously proposed in the literature~\cite{wang2012smoothed}, but it seems that no author made this central observation. 

The rest of this paper is organized as follows. The notation and definitions are introduced in Section~\ref{notations}. Section~\ref{main} contains the main result of the paper along with its proof. Finally, Section~\ref{sec:results} presents simulation results in the MRI context. 
 
\section{Notation and definitions}
\label{notations}

We shall work on the hypercube $\Omega = [0,1]^d$ with $d \geq 2$. Let $m\in \N$. The set $\Omega$ will be partitionned in $m^d$ congruent hypercubes $(\omega_i)_{i\in I}$ of edge length $1/m$.
In what follows, $\left\{ x_i \right\} _{i\in \mathbb{N}^* }$ denotes a sequence of points in the hypercube $\Omega$, independently drawn from a density $\pi:\Omega \mapsto \R_+$. The set of the first $N$ points is denoted $X_N = \left\{ x_i \right\} _{i\leqslant N}$. 

For a set of points $F$, we consider the solution to the TSP, that is the shortest Hamiltonian path between those points.
We denote $T(F)$ its length. For any set $R\subseteq \Omega $ we define $T(F, R) = T(F \cap R)$.

We also introduce $C(X_N, \Omega )$ for the optimal curve itself, and $\gamma _N: [0,1] \to \Omega $ the function that parameterizes $C(X_N, \Omega )$ by moving along it at constant speed $T(X_N, \Omega )$. 

The Lebesgue measure on an interval $[0,1]$ is denoted $\lambda _{[0,1]}$. We define the \textit{distribution of the TSP solution as follows}.

\begin{definition}
    The distribution of the TSP solution is denoted $\tilde{\Pi}_N$ and defined, for any Borelian $B$ in $\Omega $ by:
    \begin{align*}
        \tilde{\Pi}_N(B) & = \lambda _{[0,1]} \left( \gamma _N^{-1} (B) \right).
    \end{align*}
\end{definition}

\begin{remarque}
    The distribution $\tilde{\Pi}_N$ is defined for fixed $X_N$. It makes no reference to the stochastic component of $X_N$.
\end{remarque}

In order to prove the main result, we need to introduce other tools. 
For a subset $\omega_i \subseteq \Omega $, we denote the length of $C(X_N, \Omega )\cap \omega_i$ as $T_{|\omega _i}(X_N, \Omega )=T(X_N, \Omega ) \tilde{\Pi}_N(\omega _i)$. Using this definition, it follows that:
\begin{equation}
\label{eq:defalternative}
\tilde{\Pi}_N(B) = \frac{T_{|B}(X_N, \Omega )}{T(X_N,\Omega)}, \ \forall B. 
\end{equation}

Let $T_B(F, R)$ be the length of the boundary TSP on the set $F \cap R$. 
The boundary TSP is defined as the shortest Hamiltonian tour on $F \cap R$ for the metric obtained from the Euclidean metric by the quotient of the boundary of $R$, that is $d(a,b) = 0$ if $a, b \in \partial R$. Informally, it matches the original TSP while being allowed to travel along the boundary for free. We refer to~\cite{yukich_gutin2002traveling} for a complete description of this concept.

\section{Main theorem}
\label{main}

Our main theoretical result reads as follows.

\begin{theoreme}
    \label{convergence_proba}
    Let $\tilde X_N$ denote a random vector in $\R^d$ with distribution $\tilde \Pi_N$ and $X$ denote a random vector with density $\tilde \pi = \frac{\pi^{(d-1)/d}}{\int_{\Omega} \pi^{(d-1)/d}(x) dx}$.  Then $\tilde X_N$ converges in probability to $X$:
    \begin{align}
        \label{convForm}
        \tilde X_N & \stackrel{(P)}{\rightarrow} X & \mbox{$\pi^{\otimes \mathbb{N}}$-a.s.}
    \end{align}
\end{theoreme}

The remainder of this section is dedicated to proving this result.
The following proposition is central to obtain the proof:
\begin{proposition}
    \label{cube_by_cube}
    Almost surely, for all $\omega_i$ in $\{\omega_i\}_{1\leq i \leq m^d}$:    
    \begin{align}
        \label{for_small}
        \lim_{N\to \infty} \tilde{\Pi}_N(\omega _i) & = \tilde{\pi}(\omega _i) \\
                                                    & = \frac{\int_{\omega _i} \pi^{(d-1)/d}(x) \mathrm{d}x}{\int_{\Omega} \pi^{(d-1)/d}(x) \mathrm{d}x} & \mbox{$\pi^{\otimes \mathbb{N}}$-a.s.}
    \end{align}
    
\end{proposition}

The strategy consists in proving that $T_{|\omega _i}(X_N, \Omega )$ tends asymptotically to $T(X_N, \omega _i)$. The estimation of each term can then be obtained by applying the asymptotic result of Beardwood, Halton and Hammersley~\cite{beardwood1959shortest}:
\begin{theoreme}
    \label{BHH}
    If $R$ is a Lebesgue-measurable set in $\mathbb{R}^d$ such that the boundary $\partial R$ has zero measure, and $\{y_i\}_{i\in \mathbb{N}^* }$, with $Y_N = \left\{ y_i \right\} _{i\leqslant N}$ is a sequence of i.i.d. points from a density $p$ supported on $R$, then, almost surely,
    \begin{align}
        \label{BHHeq}
        \lim_{N \to \infty} \frac{T(Y_N, R) }{N^{(d-1)/d}} & = \beta (d) \int_{R}  p^{(d-1)/d}(x) \mathrm{d}x,
    \end{align}
    where $\beta (d)$ depends on the dimension $d$ only.
\end{theoreme}

We shall use a set of classical results on TSP and boundary TSP, that may be found in the survey books~\cite{yukich_gutin2002traveling}
and~\cite{yukich1998probability}.
\begin{lemme}

Let $F$ denote a set of $n$ points in $\Omega $.
\begin{enumerate}
 \item The boundary TSP is superadditive, that is, if $R_1$ and $R_2$ have disjoint interiors
\begin{align}
    \label{superadditive}
    T_B(F, R_1 \cup R_2) & \geqslant T_B(F, R_1) + T_B(F, R_2).
\end{align}

 \item The boundary TSP is a lower bound on the TSP, both globally and on subsets. If $R_2 \subset R_1$:
\begin{align}
    \label{boundT}
    T(F, R) & \geqslant T_B(F, R)
    \\
    \label{boundlocal}    
    T_{|R_2}(F, R_1) & \geqslant T_B(F, R_2)
\end{align}

  \item The boundary TSP approximates well the TSP~\cite[Lemma $3.7$]{yukich1998probability}):
\begin{align}
    \label{bound_approx}
    |T(F,\Omega) - T_B(F,\Omega)| =  O(n^{(d-2)/(d-1)} ).
\end{align}

  \item The TSP in $\Omega$ is well-approximated by the sum of TSPs in a grid of $m^d$ congruent hypercubes~\cite[Eq.~(33)]{yukich_gutin2002traveling}.
\begin{align}
    \label{approx}
    \lvert T(F, \Omega ) - \sum_{i=1}^{m^d} T(F, \omega _i) \rvert = O(n^{(d-2)/(d-1)}).
\end{align}
\end{enumerate}
\end{lemme}

We now have all the ingredients to prove the main results.
\begin{proof}[Proof of Proposition \ref{cube_by_cube}]

\begin{align*}
    \sum_{i\in I} T_B(X_N,\omega_i) & \stackrel{\eqref{superadditive}}{\leqslant} T_B(X_N,\Omega)  \\
  & \stackrel{\eqref{boundT}}{\leqslant} T(X_N,\Omega) \\
  & = \sum_{i\in I} T|_{\omega_i}(X_N,\Omega) \\
  & \stackrel{\eqref{approx}}{\leqslant} \sum_{i\in I} T(X_N,\omega_i) + O(N^{(d-1)/(d-2)})
\end{align*}

Let $N_i$ be the number of points of $X_N$ in $\omega _i$. 

Since $N_i \leqslant N$, we may use the bound \eqref{bound_approx} to get: 
\begin{equation}
\label{eq:9}
 \lim_{N\rightarrow \infty}\frac{T(X_N,\omega_i)}{N^{(d-1)/d}} =  \lim_{N\rightarrow \infty}\frac{T_B(X_N,\omega_i)}{N^{(d-1)/d}}.
\end{equation}
Using the fact that there are only finitely many $\omega _i$, the following equalities hold almost surely:
\begin{align*}
\lim_{N\rightarrow \infty}  \frac{\sum_{i\in I}T_B(X_N,\omega_i)}{N^{(d-1)/d}} 
     & = \lim_{N\rightarrow \infty} \frac{\sum_{i\in I}T(X_N,\omega_i)}{N^{(d-1)/d}} \\
     & \stackrel{\eqref{approx}}{=} \lim_{N\rightarrow \infty}  \frac{\sum_{i \in I}T_{|\omega_i}(X_N,\Omega )}{N^{(d-1)/d}}.
\end{align*}

Since the boundary TSP is a lower bound~(cf. Eqs.~\eqref{boundlocal}-\eqref{boundT}) to both local and global TSPs, the above equality ensures that:
\begin{align}
    \label{allequal}
    \lim_{N\rightarrow \infty}  \frac{T_B(X_N,\omega_i)}{N^{(d-1)/d}} & = \lim_{N\rightarrow \infty} \frac {T(X_N,\omega_i )}{N^{(d-1)/d}}\\
     & = \lim_{N\rightarrow \infty} \frac {T_{|\omega_i}(X_N,\Omega )}{N^{(d-1)/d}}   & \mbox{$\pi^{\otimes \mathbb{N} }$-a.s, $\forall i$}.\nonumber
\end{align}

Finally, by the law of large numbers, almost surely $N_i / N \to \pi(\omega _i)=\int_{\omega_i} \pi(x)dx$. 
The law of any point $x_j$ conditioned on being in $\omega _i$ has density $\pi / \pi(\omega_i)$. By applying Theorem \ref{BHH} to the hypercubes $\omega _i$ and $\Omega$ we thus get:
\begin{align*}
    \lim_{N\rightarrow +\infty} \frac{T(X_N,\omega_i)}{N^{(d-1)/d}} & = \beta(d) \int_{\omega_i} \pi(x)^{(d-1)/d}dx & \mbox{$\pi^{\otimes \mathbb{N} }$-a.s, $\forall i$}.
\end{align*}
and 
\begin{align*}
    \lim_{N\rightarrow +\infty} \frac{T(X_N,\Omega)}{N^{(d-1)/d}} & = \beta(d) \int_{\Omega} \pi(x)^{(d-1)/d}dx & \mbox{$\pi^{\otimes \mathbb{N} }$-a.s, $\forall i$}.
\end{align*}
Combining this result with Eqs.~\eqref{allequal} and \eqref{eq:defalternative} yields Proposition~\ref{cube_by_cube}.
\end{proof}

\begin{proof}[Proof of Theorem \ref{convergence_proba}]
    Let $\varepsilon > 0$ and $m$ be an integer such that $\sqrt{d} m^{-d} < \varepsilon$. Then any two points in $\omega _i$ are at distance less than $\varepsilon $.

    Using Theorem \ref{cube_by_cube} and the fact that there is a finite number of $\omega _i$, almost surely, we get:
        $\displaystyle \lim_{N\rightarrow +\infty} \sum_{i\in I} \left| \tilde{\Pi}_N(\omega_i) - \tilde \pi(\omega _i) \right| = 0$. 
 Hence, for any $N$ large enough, there is a coupling of $\tilde{\Pi}_N$ and $\tilde \pi$ such that both corresponding random variables are in the same $\omega _i$ with probability $1- \varepsilon $. Since its diameter is less than $\varepsilon $, this ends the proof.
\end{proof}

\section{Simulation results in MRI}
\label{sec:results}

The proposed sampling algorithm was assessed in a 2D MRI acquisition setup where images are sampled in the 2D Fourier
domain and compressible in the wavelet domain. Hence, $A=\mathcal{F}^*\Psi$ where $\mathcal{F}^*$ and $\Psi$ denote the discrete Fourier and inverse discrete wavelet transform, respectively.
Following~\cite{chauffert2013}, it can be shown that a near optimal sampling strategy consists of probing  $m$ independent samples of the 2D Fourier plane $(k_x,k_y)$ drawn independently from a target density $\tilde \pi$. 
A typical realization is illustrated in Fig.~\ref{fig:Distributions}(a) which in practice cannot be implemented since MRI requires probing samples along continuous curves.
To circumvent such difficulties, a TSP solver was applied to such realization in order to join all samples through a countinuous trajectory, as illustrated in Fig.~\ref{fig:Distributions}(c). Finally, Fig.~\ref{fig:Distributions}(e) shows a curve generated by a TSP solver after drawing the same amount of Fourier samples from the density $\tilde \pi^2$ as underlied by Theorem~\ref{convergence_proba}. In all sampling schemes the number of probed Fourier coefficients was equal to one fifth of the total number~(acceleration factor $r=5$).

Figs.~\ref{fig:Distributions}(b,d,f) show the corresponding reconstruction results. It is readily seen that an independent random drawing from $\tilde \pi^2$ followed by a TSP-based solver yields promising results. Moreover, a dramatic improvement of $10$dB was obtained compared to the initial drawing from $\tilde \pi$.

\begin{figure}
\begin{center}
\begin{tabular}{cc}
\yaxis{$k_y$} \figc[width=.22\textwidth]{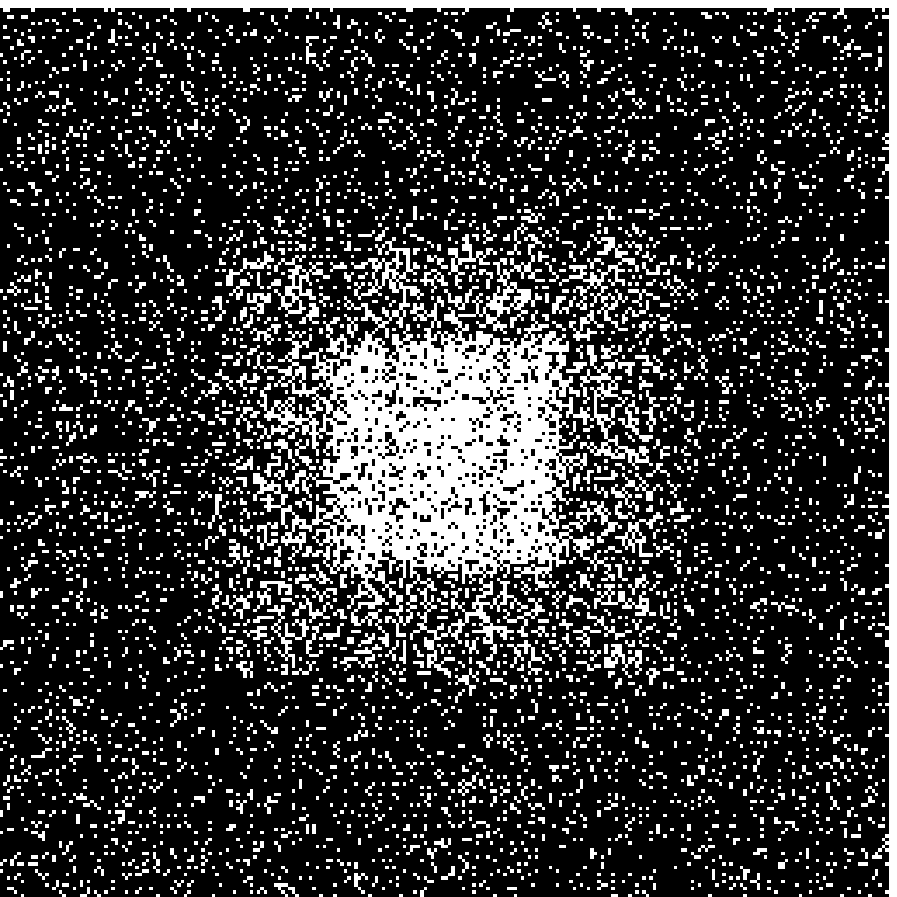}& \figc[width=.22\textwidth]{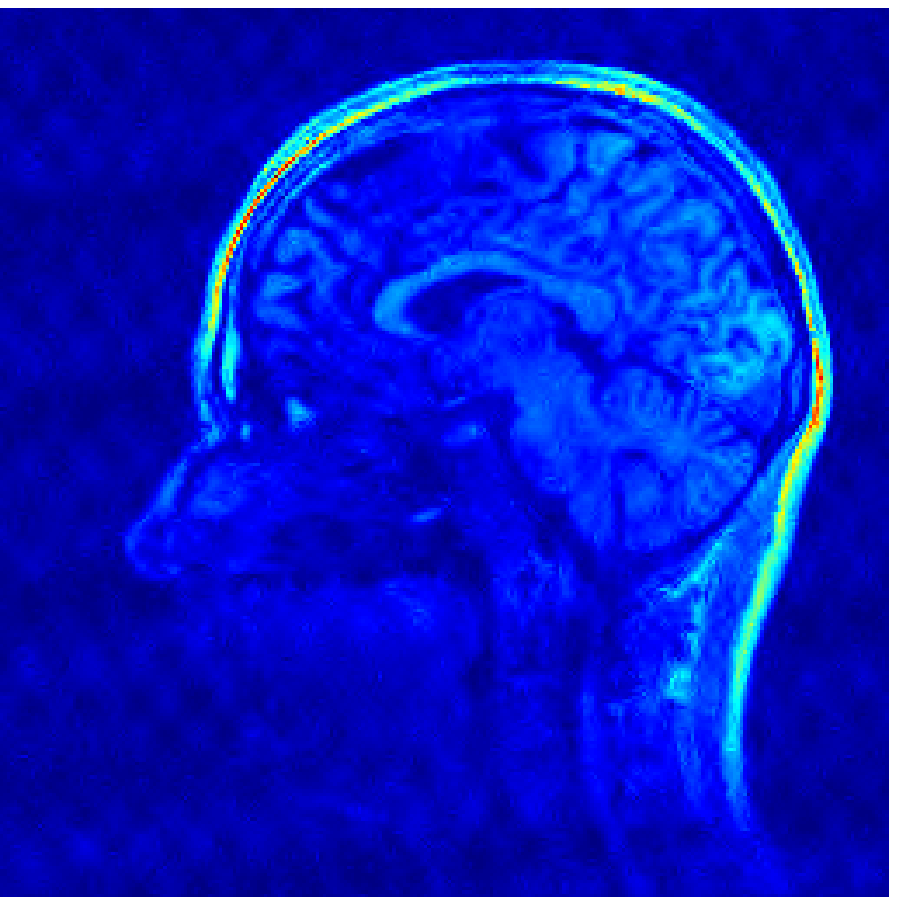}\\[-4.2cm]
{\small (a)} & {\small (b) SNR=33.0dB}\\[4.1cm]
\yaxis{$k_y$} \figc[width=.22\textwidth]{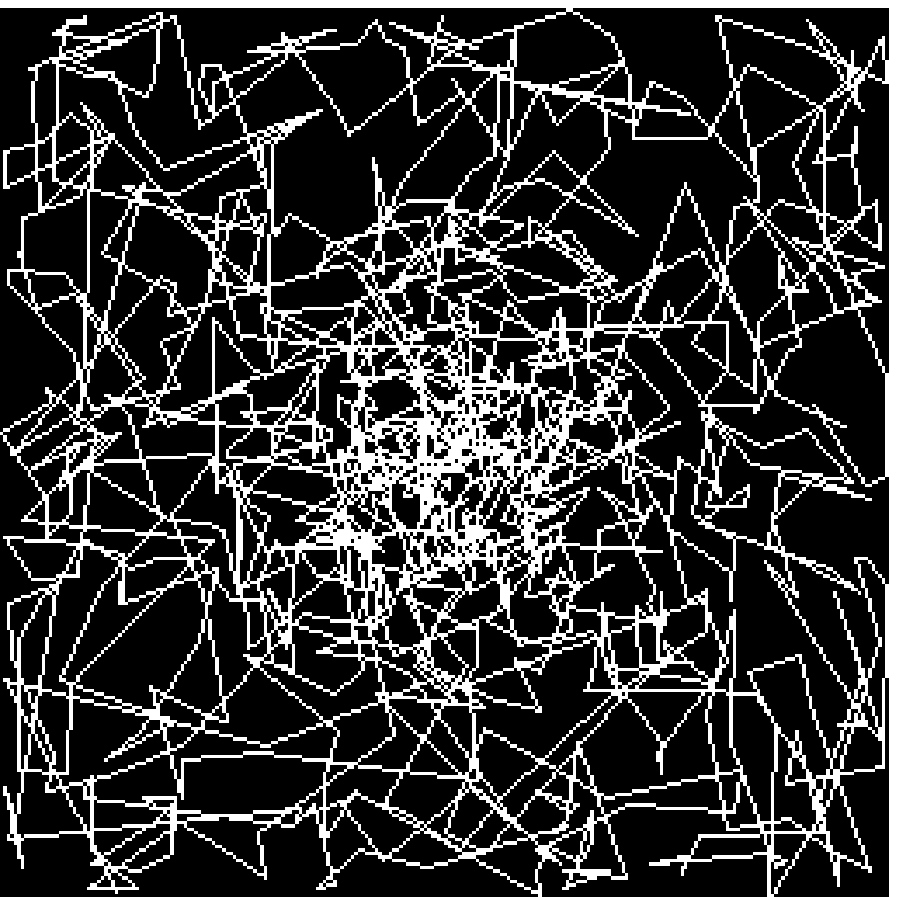}& \figc[width=.22\textwidth]{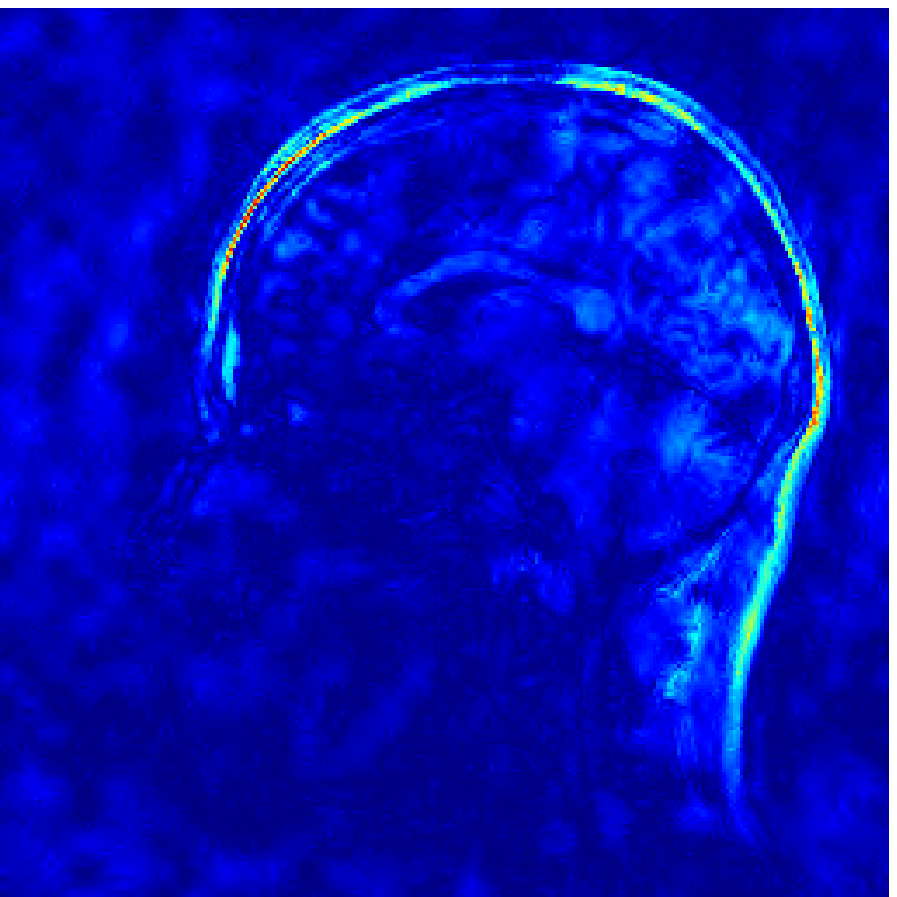}\\[-4.2cm]
{\small (c)} & {\small (d) SNR=24.1dB}\\[4.1cm]
\yaxis{$k_y$} \figc[width=.22\textwidth]{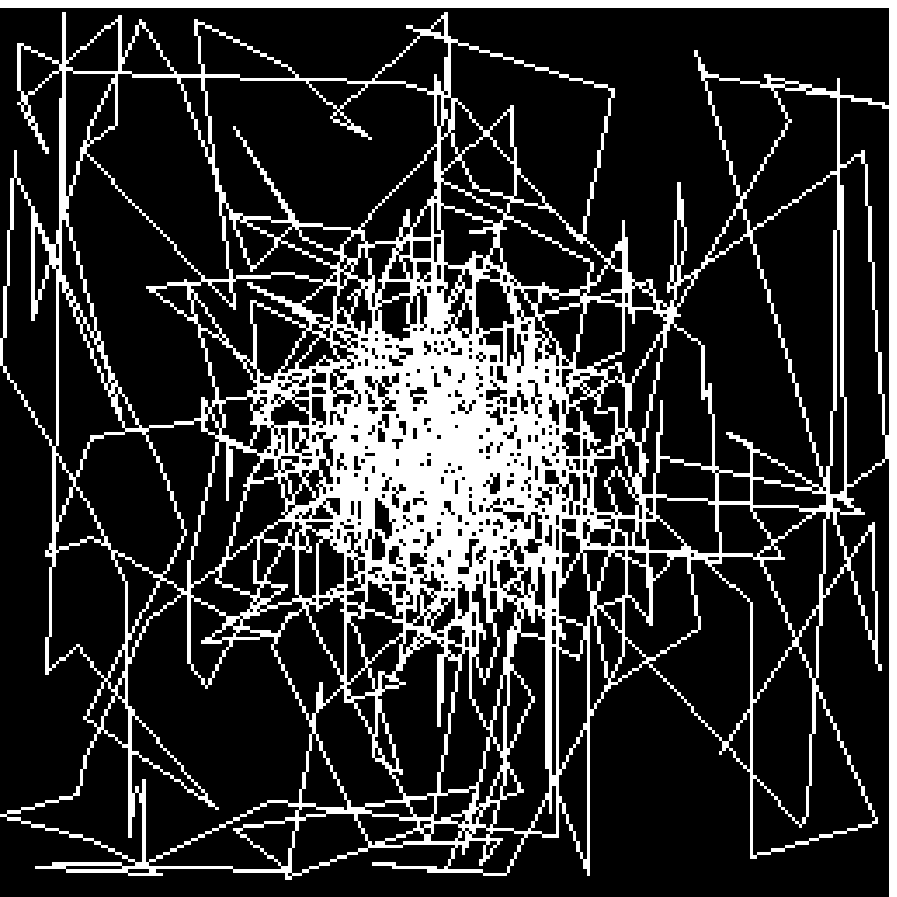} & \figc[width=.22\textwidth]{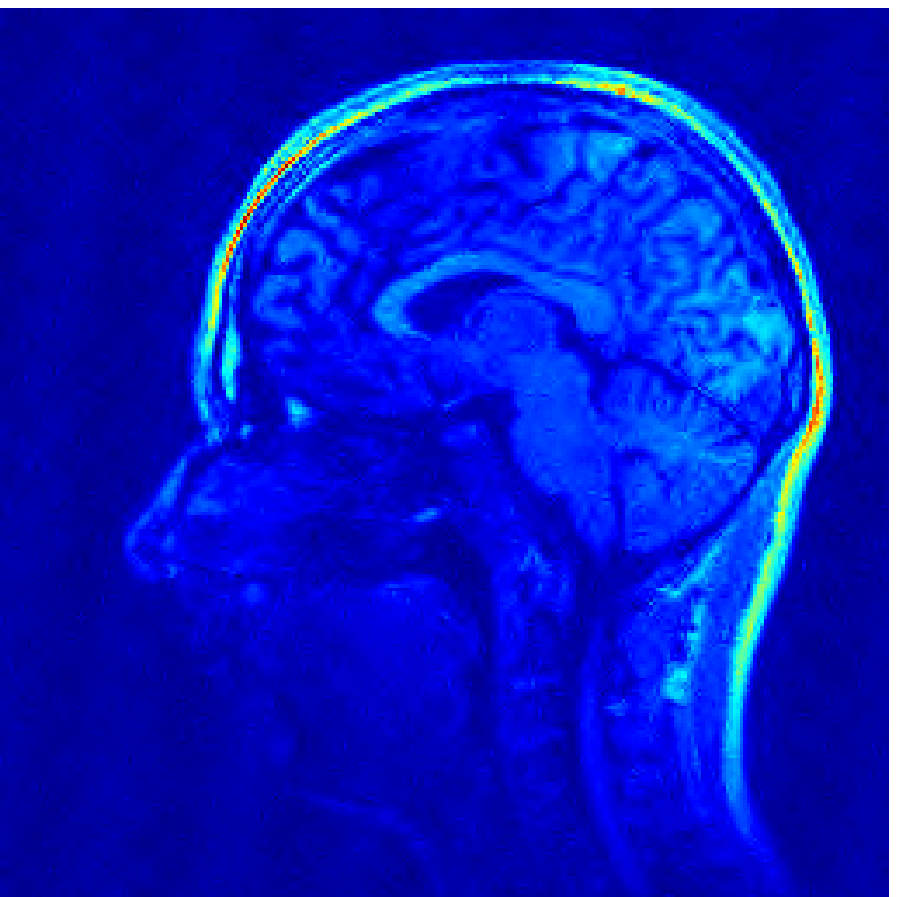}\\ 
 $k_x$ & \\[-4.6cm]
{\small (e)} & {\small (f) SNR=34.1dB}\\[4.3cm]
\end{tabular}\vspace*{-.5cm}
\end{center}
\caption{{\bf Left:} different sampling patterns (with an acceleration factor $r=5$). {\bf Right:} reconstruction results. From top to bottom: independent drawing from distribution $\tilde \pi$~(a), the same followed by a TSP solver~(c) and finally independent drawing
from distribution $\tilde \pi^2$ followed by a TSP solver.\label{fig:Distributions}}
\end{figure}

\section{Conclusion}

Designing sampling patterns lying on continuous curves is central for practical applications such as MRI.
In this paper, we proposed and justified an original two-step approach based on a TSP solver to produce such continuous trajectories. 
It allows to emulate any variable density sampling strategy and could thus be used in a large variety of applications.
In the above mentioned MRI example, this method improves the signal-to-noise ratio by $10$dB compared to more naive approaches and provides results similar to those obtained using unconstrained sampling schemes.
From a theoretical point of view, we plan to assess the convergence rate of the empirical law of the  travelling salesman trajectory to the target distribution $\pi^{(d-1)/d}$. From a practical point of view, we plan to develop algorithms that integrate stronger constraints into account such as the maximal curvature of the sampling trajectory, which plays a key role in many applications. 

\section*{Acknowledgment}

The authors would like to thank the mission pour l'interdisciplinarit\'e from CNRS and the ANR SPHIM3D for partial support of Jonas Kahn's visit to Toulouse and the CIMI Excellence Laboratory for inviting Philippe Ciuciu on a excellence researcher position during winter 2013.



\footnotesize{
\bibliographystyle{IEEEtran}
\bibliography{bibSampTA_VC}

\begin{thebibliography}{1}
\providecommand{\url}[1]{#1}
\csname url@samestyle\endcsname
\providecommand{\newblock}{\relax}
\providecommand{\bibinfo}[2]{#2}
\providecommand{\BIBentrySTDinterwordspacing}{\spaceskip=0pt\relax}
\providecommand{\BIBentryALTinterwordstretchfactor}{4}
\providecommand{\BIBentryALTinterwordspacing}{\spaceskip=\fontdimen2\font plus
\BIBentryALTinterwordstretchfactor\fontdimen3\font minus
  \fontdimen4\font\relax}
\providecommand{\BIBforeignlanguage}[2]{{%
\expandafter\ifx\csname l@#1\endcsname\relax
\typeout{** WARNING: IEEEtran.bst: No hyphenation pattern has been}%
\typeout{** loaded for the language `#1'. Using the pattern for}%
\typeout{** the default language instead.}%
\else
\language=\csname l@#1\endcsname
\fi
#2}}
\providecommand{\BIBdecl}{\relax}
\BIBdecl

\bibitem{candes2006near}
E.~Candes and T.~Tao, ``Near-optimal signal recovery from random projections:
  Universal encoding strategies?'' \emph{Information Theory, IEEE Transactions
  on}, vol.~52, no.~12, pp. 5406--5425, 2006.

\bibitem{rauhut2010compressive}
H.~Rauhut, ``Compressive sensing and structured random matrices,''
  \emph{Theoretical foundations and numerical methods for sparse recovery},
  vol.~9, pp. 1--92, 2010.

\bibitem{lustig2007sparse}
M.~Lustig, D.~Donoho, and J.~Pauly, ``Sparse {MRI}: The application of
  compressed sensing for rapid mr imaging,'' \emph{Magnetic Resonance in
  Medicine}, vol.~58, no.~6, pp. 1182--1195, 2007.

\bibitem{knoll2011adapted}
F.~Knoll, C.~Clason, C.~Diwoky, and R.~Stollberger, ``Adapted random sampling
  patterns for accelerated {MRI},'' \emph{Magnetic Resonance Materials in
  Physics, Biology and Medicine}, vol.~24, no.~1, pp. 43--50, 2011.

\bibitem{chauffert2013}
N.~Chauffert, P.~Ciuciu, and P.~Weiss, ``Variable density compressed sensing in
  {MRI}. {T}heoretical {VS} heuristic sampling strategies.'' in
  \emph{proceedings of IEEE ISBI}, 2013.

\bibitem{wang2012smoothed}
H.~Wang, X.~Wang, Y.~Zhou, Y.~Chang, and Y.~Wang, ``Smoothed random-like
  trajectory for compressed sensing {MRI},'' in \emph{Engineering in Medicine
  and Biology Society (EMBC), 2012 Annual International Conference of the
  IEEE}, 2012, pp. 404--407.

\bibitem{yukich_gutin2002traveling}
A.~M. Frieze and J.~E. Yukich, ``Probabilistic analysis of the tsp,'' in
  \emph{The traveling salesman problem and its variations}, ser. Combinatorial
  optimization, G.~Gutin and A.~P. Punnen, Eds.\hskip 1em plus 0.5em minus
  0.4em\relax Springer, 2002, vol.~12, pp. 257--308.

\bibitem{beardwood1959shortest}
J.~Beardwood, J.~Halton, and J.~Hammersley, ``The shortest path through many
  points,'' in \emph{Mathematical Proceedings of the Cambridge Philosophical
  Society}, vol.~55, no.~04.\hskip 1em plus 0.5em minus 0.4em\relax Cambridge
  Univ Press, 1959, pp. 299--327.

\bibitem{yukich1998probability}
J.~Yukich, \emph{Probability theory of classical Euclidean optimization
  problems}.\hskip 1em plus 0.5em minus 0.4em\relax Springer, 1998.

\end{thebibliography}
}
%
%


\end{document}